\documentclass[11pt,a4paper, reqno, oneside]{amsart}

\usepackage{amsmath,amssymb,amsthm}

\numberwithin{equation}{section}
\newtheorem{theorem}{Theorem}[section]

\newtheorem{lemma}[theorem]{Lemma}
\theoremstyle{remark}
\newtheorem{remark}{Remark}[section]

\theoremstyle{definition}
\newtheorem{definition}[theorem]{Definition}
\newtheorem{assumption}[theorem]{Assumption}
\newtheorem{example}[theorem]{Example}

\newcommand{\R}{\mathbb{R}}
\newcommand{\C}{\mathbb{C}}

\begin{document}

\title[Schr\"odinger operators with non-trapping magnetic fields]
{Non-trapping magnetic fields and Morrey-Campanato estimates for
Schr\"odinger operators}

\begin{abstract}
  We prove some uniform in $\epsilon$ a priori estimates for solutions of the
  equation
  \begin{equation*}
    (\nabla-iA)^2u-V(x)u+(\lambda\pm i\epsilon)u=f,
    \qquad
    \lambda\geq0,
    \quad
    \epsilon\neq0.
  \end{equation*}
  The estimates are obtained in terms
  of Morrey-Campanato norms, and can be used to prove absence of
  zero-resonances, in a suitable sense, for electromagnetic
  Hamiltonians. Precise conditions on the size of the
  \textit{trapping component} of the magnetic field and the
  non repulsive component of the electric field are given.
\end{abstract}

\date{\today}

\author{Luca Fanelli}
\address{Luca Fanelli: Universidad del Pais Vasco, Departamento de
Matem$\acute{\text{a}}$ticas, Apartado 644, 48080, Bilbao, Spain}
\email{luca.fanelli@ehu.es}



\maketitle


\section{Introduction}\label{sec.introd}
In space dimension $n\geq3$, let us consider the electromagnetic
Schr\"odinger operator
\begin{equation}\label{eq:hamiltonian}
  H=-(\nabla-iA(x))^2+V(x);
\end{equation}
here $A=(A^1,\cdots A^n):\R^n\to\R^n$ is the magnetic potential, and
$V:\R^n\to\R$ is the electric potential. We denote by
\begin{equation*}
  \nabla_A=\nabla-iA,
  \qquad
  \Delta_A=\nabla_A^2.
\end{equation*}
In the theory of electromagnetic fields, a deep literature has been
produced on the study of electromagnetic Schr\"odinger Hamiltonians
\eqref{eq:hamiltonian}. There are indeed a lot of interesting
problems related to the properties of solutions of stationary and
evolutive equations described by these operators. The magnetic
potential $A$ is a mathematical construction which describes the
interaction of particles with an external magnetic field. The vector
field $A$ is standardly associated to a 1-form, whose differential
$B:=dA$ is the magnetic field, which is a physical object. We can
define analytically $B$ as the $n\times n$ anti-symmetric matrix
\begin{equation*}
  B=DA-(DA)^t,
  \qquad
  (DA)_{ij}=\frac{\partial A^i}{\partial x_j},
  \quad
  (DA)^t_{ij}=(DA)_{ji}.
\end{equation*}
In dimension $n=3$, the magnetic field $B$ is identified as
$B=\text{curl}A$, due to the isomorphism between 1-forms and
2-forms; this fact has to be interpreted in terms of the action
\begin{equation*}
  Bv=\text{curl}A\times v,
  \qquad
  \text{for all }v\in\R^3,
\end{equation*}
where the cross is the vectorial product on $\R^3$. We will always
consider smooth potentials $A,V\in\mathcal C^1$; actually, it is
possible to study the validity of the results of this paper for
rough potentials, but it is not in our aims. Moreover, in what
follows, we always assume:
\begin{assumption}\label{ass:H}
  the Hamiltonian $H$ is self-adjoint on $L^2(\R^n)$, with
  form-domain
  \begin{equation*}
    \mathcal D(H)=\{f\in
    L^2(\R^n):\int|\nabla_Af|^2+\int|V|\cdot|f|^2<\infty\}.
  \end{equation*}
\end{assumption}
Assumption \ref{ass:H} has several consequences: the spectrum
$\sigma(H)$ is real, and via Spectral Theorem we can perform the
functional calculus $g(H)$, for any Borel-measurable function $g$.
In particular, by the powers of the operator $H$ we can define the
distorted Sobolev norms
\begin{equation*}
  \|f\|_{\dot{\mathcal H}^s}=\|H^{\frac s2}f\|_{L^2}.
\end{equation*}
The validity of Assumption \ref{ass:H} requires local integrability
conditions on $A,V$, and the literature about it is complete. For
details, see the Leinfelder-Simader result in \cite{LS} and the book
\cite{CFKS}.

The aim of this paper is to prove uniform (in $\epsilon$) a priori
estimates for solutions of the resolvent equation
\begin{equation}\label{eq:helmholtz}
  -Hu(x)+(\lambda\pm i\epsilon)u(x)=f(x),
  \qquad
  \lambda\geq0,
  \quad
  \epsilon\neq0
\end{equation}
by direct techniques based on integration by parts. In the purely
electric case $A\equiv0$, we shall mention \cite{pv} as inspirator
of this multipliers technique (actually the subject there is the
Helmholtz equation, and the role of $V$ is played by the
rarefraction index $n(x)$). Since $\lambda\pm i\epsilon\notin\R$,
for any $f$ in $L^2$ there exists a unique $u\in L^2$ solution of
\eqref{eq:helmholtz}.

The integration by parts gives very precise informations about the
relevant quantities (related to the electromagnetic field) which
play a role in the spectral properties of $H$. It is of particular
interest the part concerning the magntic potential $A$. Let us give
the following definition.
\begin{definition}[non trapping magnetic fields]\label{def.bitau}
  Let us define by $B_\tau:\R^n\to\R^n$ the tangential component of the magnetic
  field $B$, given by
  \begin{equation*}
    B_\tau(x):=\frac{x}{|x|}B.
  \end{equation*}
  Observe that in dimension $n=3$ it coincides with
  \begin{equation*}
    B_\tau(x):=\frac{x}{|x|}\times\text{curl}A(x).
  \end{equation*}
  We say that $B$ is \textit{non-trapping} if $B_\tau=0$.
\end{definition}
The quantity $B_\tau$ was introduced in \cite{FV}, in which it is
proved that weak-dispersion for the magnetic Schr\"odinger and wave
equation holds, for example, for non-trapping potentials. Indeed, a
smallness condition on $B_\tau$ is sufficient there to prove that
some aspects of the free dynamics are preserved in presence of this
kind of fields. This is also what happens in the stationary case, as
we prove later in our main theorems. We give some examples of
non-trapping fields (see also \cite{FV}), in dimension $n=3$.
\begin{example}\label{ex.coulomb}
 Let us take
  \begin{equation}\label{eq.example}
    A=\frac1{x^2+y^2+z^2}(-y,x,0)=\frac{1}{x^2+y^2+z^2}(x,y,z)\times(0,0,1).
  \end{equation}
  One can easily check that
  \begin{equation*}
    \nabla\cdot A=0,
    \qquad
    B=-2\frac{z}{(x^2+y^2+z^2)^2}(x,y,z),
    \qquad
    B_\tau=0.
  \end{equation*}
  Another (more singular) example is the following:
  \begin{equation}\label{eq.example2}
    A=\left(\frac{-y}{x^2+y^2},\frac{x}{x^2+y^2},0\right)=
    \frac{1}{x^2+y^2}(x,y,z)\times(0,0,1).
  \end{equation}
  Here we have $B=(0,0,\delta)$, with $\delta$ denoting Dirac's delta function. Again we have $B_\tau=0$ .
\end{example}
\begin{example}\label{ex.biot}
  A natural generalization of the previous examples is the following
  one.
  Assume that $B=\text{curl}\,A:\R^3\to\R^3$ is known; if we fix the Coulomb gauge $\text{div}A=0$, then
  $A$ can be obtained by the {\it Biot-Savart}
  formula
  \begin{equation}\label{eq.BS}
    A(x)=\frac1{4\pi}\int\frac{x-y}{|x-y|^3}\times B(y)\,dy.
  \end{equation}
  Let us assume $B_\tau=0$, namely $x\times B(x)=0$; by
  \eqref{eq.BS} we have
  \begin{equation}\label{eq.BS2}
    A(x)=\frac x{4\pi}\times\int\frac{B(y)}{|x-y|^3}\,dy.
  \end{equation}
  Consequently, for the condition $B_\tau=0$ it is necessary $B(y)=
  g(y)\frac{y}{|y|}$, for some $g:\R^3\to\R$. Since we want
  $A\neq0$, $g$ has not to be radial. For example we
  consider
  \begin{equation*}
    g(y)=h\left(\frac{y}{|y|}\cdot\omega\right)|y|^{-\alpha},
  \end{equation*}
  for some fixed $\omega\in S^2$, where $h$ is homogeneous of degree
  0 and $\alpha\in\R$; as a consequence, the vector field $B$ is homogeneous of degree
  $-\alpha$. By \eqref{eq.BS2} we have
  \begin{equation}\label{eq.BS3}
    A(x)=\frac x{4\pi}\times\int\frac{h\left(\frac{y}{|y|}\cdot\omega\right)}
    {|x-y|^3|y|^\alpha}y\,dy.
  \end{equation}
  The potential $A$ is homogenous of degree $1-\alpha$, and by
  symmetry we have that $A(\omega)=0$. These examples can be easily extended to higher dimensions.
\end{example}
Before stating the main theorems, we need to introduce some
notations. For $f:\R^n\to\C$ we define the Morrey-Campanato norm as
\begin{equation*}
  |||f|||^2=\sup_{R>0}\frac1R\int_{|x|\leq R}|f|^2dx.
\end{equation*}
Moreover, we denote by $C(j)=\{x\in\R^n:2^j\leq|x|\leq2^{j+1}\}$,
\begin{equation*}
  N(f)=\sum_{j\in\mathbb
  Z}\left(2^{j+1}\int_{C(j)}|f|^2dx\right)^{\frac12},
\end{equation*}
and we easily notice the duality relation
\begin{equation*}
  \int fgdx\leq|||g|||\cdot N(f).
\end{equation*}
For any $p\geq1$, we also define
\begin{equation*}
  \|f\|_{L^p_rL^\infty(S_r)}=\left(\int_0^{+\infty}\sup_{|x|=r}|f(x)|^pdr\right)^{\frac1p}.
\end{equation*}
We are now ready to state our main results.
\begin{theorem}[3D-Morrey-Campanato estimates]\label{thm:main}
  Let $n=3$; let us assume that
  \begin{equation}\label{eq:btau}
    \||x|^{\frac32}B_\tau\|_{L^2_rL^\infty(S_r)}=C_1<\infty
  \end{equation}
  \begin{equation}\label{eq:vr}
    \||x|^2(\partial_rV)_+\|_{L^1_rL^\infty(S_r)}=C_2<\infty
  \end{equation}
  \begin{equation}\label{eq:v+}
    \|\langle x\rangle^{-1}|x|^2V_+\|_{L^1_rL^\infty(S_r)}=C_3<\infty,
  \end{equation}
  and moreover there exists $M\geq0$ such that
  \begin{equation}\label{eq:condition}
    \frac{\left(M+\frac12\right)^2}{M}C_1^2+2\left(M+\frac12\right)C_2<1.
  \end{equation}
  Assume, moreover, that $V$ satisfies the Hardy-type condition
  \begin{equation}\label{eq:hardy}
    \int|V|\cdot|u|^2dx\leq C\int|\nabla_Au|^2dx,
  \end{equation}
  for some $C>0$. Then, any solution $u\in\mathcal H^1$ of equation
  \eqref{eq:helmholtz} satisfies the following a priori estimates:
  \begin{align}\label{eq:estimate3D}
  & |||\nabla_Au|||^2
  +|u(0)|^2
  +\frac M2\int(\partial_rV)_-|u|^2
  \\
  &
  +\delta\left(\int\langle x\rangle^{-1}V_-|u|^2
  +\lambda\int\frac{|u|^2}{\langle x\rangle}
  +
  \int\frac{|\nabla_A^\tau u|^2}{|x|}
  +\sup_{R>0}\frac1{R^2}\int_{|x|=R}|u|^2
  d\sigma
  \right)
  \nonumber
  \\
  &\ \
  \leq
  C\left[N(f)^2+(|\epsilon|+\lambda)
  \left(N\left(\frac{f}{|\lambda|^{1/2}}\right)\right)^2\right],
  \nonumber
\end{align}
for some $C>0$ and some small $\delta>0$ depending on
$C_1,C_2,C_3,M$.
\end{theorem}
\begin{theorem}[Higher-dimensional Morrey-Campanato
estimates]\label{thm:main2}
  Let $n\geq4$ and; let us assume that
  \begin{equation}\label{eq:btau4D}
    \||x|^2B_\tau\|_{L^\infty}=C_1<\infty
  \end{equation}
  \begin{equation}\label{eq:vr4D}
    \||x|^3(\partial_rV)_+\|_{L^\infty}=C_2<\infty
  \end{equation}
  \begin{equation}\label{eq:v+4D}
    \|\langle x\rangle^{-1}|x|^3V_+\|_{L^\infty}=C_3<\infty,
  \end{equation}
  and moreover
  \begin{equation}\label{eq:condition2}
    C_1^2+2C_2<(n-1)(n-3).
  \end{equation}
  Assume, moreover, that $V$ satisfies the Hardy-type condition
  \begin{equation}\label{eq:hardy2}
    \int|V|\cdot|u|^2dx\leq C\int|\nabla_Au|^2dx,
  \end{equation}
  for some $C>0$. Then, any solution $u\in\mathcal H^1$ of equation
  \eqref{eq:helmholtz} satisfies the following a priori estimates:
  \begin{align}\label{eq:estimate4D}
  & |||\nabla_Au|||^2
  +\sup_{R>0}\left(\frac{1}{R^{2}}\int_{|x|=R}|u|^2d\sigma\right)
  +\int(\partial_rV)_-|u|^2
  \\
  &\ \ \
  +\delta\left(\int\langle x\rangle^{-1}V_-|u|^2
  +\lambda\int\frac{|u|^2}{\langle x\rangle}+
  \int\frac{|\nabla_A^\tau u|^2}{|x|}
  +\int\frac{|u|^2}{|x|^3}dx
  \right)
  \nonumber
  \\
  &\ \
  \leq
  C\left[N(f)^2+(|\epsilon|+\lambda)
  \left(N\left(\frac{f}{|\lambda|^{1/2}}\right)\right)^2\right],
  \nonumber
\end{align}
for some $C>0$ and some small $\delta>0$ depending on
$C_1,C_2,C_3,M$.
\end{theorem}
Let us make some remarks about the statements of Theorems
\ref{thm:main}, \ref{thm:main2} and their possible applications.
\begin{remark}\label{rem:1}
  Estimates \eqref{eq:estimate3D} and \eqref{eq:estimate4D} recover
  the uniform (with respest ti $\epsilon$) estimate in the main Theorem of \cite{pv}, in the
  purely electric case $A\equiv0$ (actually the refraction index
  $n(x)$ there plays the role of our electric potential $V$). In
  fact, here we have some gain in the term involving $\lambda$ at
  the left-hand side (analogous to the term $|||n^{1/2}u|||^2$ in the main Theorem by
  \cite{pv}), which is due to an appropriate choice of the symmetric
  multiplier $\varphi$ (see Section \ref{sec:proof} in the
  following).
\end{remark}
\begin{remark}[Assumptions on the electromagnetic
field]\label{rem:assu}
  Let us give an interpretation of assumptions \eqref{eq:condition},
  \eqref{eq:condition2}.
  Observe the difference on the decay and singularity informations about
  $A,V$, between the 3D case and the higher dimensional case.
  Indeed, in dimension $n=3$, potentials behaving like $|A|=C/|x|$,
  $|V|=C/|x|^2$ are not allowed, while assumptions \eqref{eq:btau},
  \eqref{eq:vr}, \eqref{eq:v+} are satisfied by potentials with
  these behaviors
  \begin{equation*}
    |A|\leq\frac{C}{|x|^{1-\epsilon}+|x|^{1+\epsilon}},
    \qquad
    |V|\leq\frac{C}{|x|^{2-\epsilon}+|x|^{2+\epsilon}},
  \end{equation*}
  with $C>0$, and according with the smallness of $B_\tau$ and
  $(\partial_rV)_+$ required by \eqref{eq:condition}. In fact,
  potentials with critical decay and singularity are permitted by
  the higher dimensional assumptions \eqref{eq:btau4D},
  \eqref{eq:vr4D}, \eqref{eq:v+4D} and \eqref{eq:condition2}.

  Moreover, notice that the size of $C_3$ is not
  relevant, both in \eqref{eq:condition} and \eqref{eq:condition2};
  indeed, no smallness assumption on $V$ is needed in
  order to obtain estimates \eqref{eq:estimate3D},
  \eqref{eq:estimate4D}. In the 3D case, assume that $C_1=0$, i.e. the field $B$
  is non-trapping, according to Definition \ref{def.bitau}; hence,
  since $\min_{M\geq0}2(M+1/2)=1$, condition \eqref{eq:condition}
  simply reads
  \begin{equation*}
    C_2<1.
  \end{equation*}
  On the other hand, if we assume $C_2=0$, in other words $V$ is
  repulsive, since $\min_{M\geq0}(M+1/2)^2/M=2$, the condition on
  $C_1$ is
  \begin{equation*}
    C_1^2<\frac12.
  \end{equation*}
  We claim that \eqref{eq:condition} is in fact sharp; it would be
  interesting to find counterexamples to estimate
  \eqref{eq:estimate3D}, with potentials satisfying \eqref{eq:btau},
  \eqref{eq:vr} and \eqref{eq:v+}, but not satisfying
  \eqref{eq:condition}.

  Observe also that no assumptions on $A$ (except for the
  self-adjointness) are in the statement of Theorems \ref{thm:main},
  \ref{thm:main2}; hence the gauge invariance of these results is
  preserved.
\end{remark}
\begin{remark}[Hardy conditions on $V$]\label{rem:hardy}
  The Hardy-type conditions \eqref{eq:hardy}, \eqref{eq:hardy2} have
  to be interpreted by means of the magnetic Hardy inequality
  \begin{equation}\label{eq:hardymagn}
    \int_{\R^n}\frac{|u|^2}{|x|^2}dx\leq\frac{(n-2)^2}{4}\int_{\R^n}|\nabla_Au|^2dx,
  \end{equation}
  which holds in dimension $n\geq3$ on any function $u\in\mathcal
  H^1$ (see \cite{FV} for a simple proof of \eqref{eq:hardymagn} by integration by
  parts).
\end{remark}
\begin{remark}[Absence of resonances]\label{rem:resonance}
  One of the possible applications of Theorems \ref{thm:main} and
  \ref{thm:main2} is to prove absence of zero-energy resonances
  for the Hamiltonian $H$. Actually the right definition of
  resonances is not completely clear, see e.g. \cite{a}, \cite{ah},
  \cite{is},
  \cite{j}, \cite{j2}, \cite{jk}, \cite{mu}, \cite{rs}. In fact,
  in the study of dispersive equations
  related to $H$, as the magnetic Schr\"odinger equation
  \begin{equation*}
    iu_t+Hu=0,
  \end{equation*}
  or the magnetic wave equation
  \begin{equation*}
    u_{tt}+Hu=0,
  \end{equation*}
  a typical abstract assumption of absence of zero-energy resonances
  is needed, in order to preserve the free dynamics (see e.g. the
  recent papers \cite{ges}, \cite{ges2}, \cite{pda-lf1},
  \cite{pda-lf2}, \cite{ste}). By the statement of Theorems \ref{thm:main} and
  \ref{thm:main2} it is natural to consider the following definition
  of zero-energy resonances, introduced in \cite{brv}:
  \begin{definition}\label{def:res}
    A function $u$ is a zero-resonance if
    \begin{equation*}
      u\notin L^2,
      \qquad
      u\in\mathcal H^1_{\text{loc}},
      \qquad
      |V|^{\frac12}u\in L^2,
    \end{equation*}
    \begin{equation*}
      \sup_{R>1}\frac1R\int_{|x|\leq R}\left[|V|+\langle
      x\rangle^{-2}\right]|u|^2<\infty,
    \end{equation*}
    \begin{equation*}
      \liminf_{R\to\infty}\frac1R\int_{|x|\leq R}\left[|V|+\langle
      x\rangle^{-2}\right]|u|^2=0,
    \end{equation*}
    and $u$ satisfies the equation
    \begin{equation*}
      -Hu=0.
    \end{equation*}
  \end{definition}
  It is possible to see that Theorems \ref{thm:main} and
  \ref{thm:main2} imply the absence of zero-energy resonances,
  according to the previous Definition. Indeed, one should repeat
  the proof of Lemma \ref{lem:parts} (see Section
  \ref{sec:morawetz}) by performing the integration by parts on
  compact balls of $\R^n$, taking into account the boundary terms,
  which in fact turn out with correct signs. We omit here
  further details (for completeness we remand to \cite{brv}, final
  section). Actually, we also remark that in \cite{FV} it is proved that, under
  assumptions of type \eqref{eq:condition},
  \eqref{eq:condition2}, weakly dispersive estimates and Strichartz
  estimates are true for the magnetic Schr\"odinger and wave
  equation;
  in that paper, no abstract assumptions on the Hamiltonian
  (namely absence of zero-resonances) are needed. Observe that our
  assumptions on the term $B_\tau$ does not appear in \cite{is}, in
  which First order perturbations of $-\Delta$ are also treated.
\end{remark}
The rest of the paper is devoted to the proofs of the main theorems.
These are based on the Morawetz-type Lemma \ref{lem:parts}, which is
proved in the next section; then a suitable choice of the
multipliers (last section) completes the proofs.

\section{Integration by parts}\label{sec:morawetz}

In this Section we state and prove Lemma \ref{lem:parts}, which is
our fundamental tool for the proof of the main theorems. It is based
on the standard technique of Morawetz multipliers, introduced in
\cite{mor} for the Klein-Gordon equation and then used in several
other contests (dispersive equations, kinetic equations, Helmholtz
equations ecc...). We should mention here \cite{pv}, as a seminal
work about the relation between Morawetz methods and
Morrey-Campanato estimates for the Helmholtz equation. Later, in
\cite{brv}, \cite{FV} it was shown as these techniques can be
adopted to prove some weak-dispersive estimates for Schr\"odinger
and wave equations with electric and electromagnetic potentials.

We prove the following Lemma, which will be used to prove the main
theorems.
\begin{lemma}\label{lem:parts}
  Let $\phi(|x|),\psi(|x|)$ be two radial, real-valued multipliers and let
  $u\in\mathcal H^1$ be a solution of equation \eqref{eq:helmholtz}. Then, the
  following identity holds:
  \begin{align}\label{eq:identity}
    & \int\nabla_AuD^2\phi\overline{\nabla_Au}dx
    -\int\varphi\left|\nabla_Au\right|^2dx
    -\int\left(\frac14\Delta^2\phi-\frac12\Delta\varphi\right)|u|^2dx
    \\
    & -\int\left[\frac12\phi'(\partial_rV)+\varphi V\right]|u|^2dx
    +\Im\int\phi'uB_\tau\cdot\overline{\nabla_Au}dx
    +\lambda\int\varphi|u|^2dx
    \nonumber
    \\
    & = \Re\int
    f\left(\nabla\phi\cdot\overline{\nabla_Au}+\frac12(\Delta\phi)\overline{u}\right)dx
    +\Re\int f\varphi\overline udx
    \pm\epsilon\Im\int u\nabla\phi\cdot\overline{\nabla_A u}dx,
    \nonumber
  \end{align}
  where $D^2\phi,\Delta^2\phi$ denote, respectively, the Hessian and
  the bi-Laplacian of $\phi$, while $B_\tau$ is as in Definition
  \ref{def.bitau}.
\end{lemma}
\begin{proof}
  We
  divide the proof into two parts, acting on equation \eqref{eq:helmholtz}
  with a symmetric multiplier first, and then with an anti-symmetric
  one.

  {\bf Symmetric multiplier.} Let us multiply equation
  \eqref{eq:helmholtz} by $\varphi u$ in the $L^2$-sense; taking
  the resulting real parts, and observing that
  \begin{equation*}
    -\Re(Hu,\varphi
    u)_{L^2}=-\int\varphi\left|\nabla_Au\right|^2dx
    +\frac12\int\Delta\varphi|u|^2dx
    -\int\varphi V|u|^2dx,
  \end{equation*}
  it gives the identity
  \begin{align}\label{eq:symmetric}
    & -\int\varphi\left|\nabla_Au\right|^2dx
    +\frac12\int\Delta\varphi|u|^2dx
    -\int\varphi V|u|^2dx
    +\lambda\int\varphi|u|^2dx
    \\
    & \ \ \ =\Re\int f\varphi\overline udx.
    \nonumber
  \end{align}
  On the other hand, the imaginary parts give
  \begin{equation}\label{eq:symmetric2}
    \pm\epsilon\int\varphi|u|^2dx=\Im\int f\varphi\overline udx.
  \end{equation}

  {\bf Anti-symmetric multiplier.} Let us multiply equation
  \eqref{eq:helmholtz} by
  \begin{equation*}
    \frac12[H,\phi]u=\nabla\phi\cdot\nabla_Au+\frac12(\Delta\phi)u,
  \end{equation*}
  in the sense of $L^2$. It gives
  \begin{equation}\label{eq:1}
    -\frac12(Hu,[H,\phi]u)_{L^2}+\frac\lambda2(u,[H,\phi]u)_{L^2}
    \pm i\frac{\epsilon}{2}(u,[H,\phi]u)_{L^2}=
    (f,[H,\phi]u)_{L^2}.
  \end{equation}
  Now we take the real part of identity \eqref{eq:1}. First observe
  that, since the commutator $[H,\phi]$ is anti-symmetric, we have
  \begin{equation}\label{eq:2}
    \Re(u,[H,\phi]u)_{L^2}=0,
  \end{equation}
  \begin{equation}\label{eq:22}
    \pm\Re i\frac\epsilon2(u,[H,\phi]u)_{L^2}=\mp\epsilon\Im\int u\nabla\phi\cdot\overline{\nabla_A
    u}dx.
  \end{equation}
  For the same reason, we see immediately that
  \begin{equation}\label{eq:3}
    -\frac12\Re(Hu,[H,\phi]u)_{L^2}=
    -\frac14([H,[H,\phi]]u,u)_{L^2}.
  \end{equation}
  The explicit computation of the second commutator $[H,[H,\phi]]$
  has been already performed in \cite{FV}; this is the point in which the trapping component $B_\tau$
  appears.  By see formulas (1.13) and (2.3) in \cite{FV} we obtain
  \begin{align}\label{eq:4}
    -\frac14([H,[H,\phi]]u,u)_{L^2}= &
    \int\nabla_AuD^2\phi\overline{\nabla_Au}dx-\frac14\int|u|^2\Delta^2\phi dx
    \\
    & -\frac12\int\phi'\partial_rV|u|^2dx
    +\Im\int\phi'uB_\tau\cdot\overline{\nabla_Au}dx.
  \end{align}
  We remark that the idea of the computation \eqref{eq:4} in \cite{FV} is to use the
  Leibnitz formula for $\nabla_A$ in the form $\nabla_A(fg)=(\nabla_Af)f+(\nabla
  g)f$; hence we can put all the distorted derivatives on the
  solution and the straight derivatives on the multiplier.

  Finally, by \eqref{eq:1}, \eqref{eq:2}, \eqref{eq:22}, \eqref{eq:3} and \eqref{eq:4}
  we obtain the following identity:
  \begin{align}\label{eq:anti-symmetric}
    &
    \int\left(\nabla_AuD^2\phi\overline{\nabla_Au}-\frac14|u|^2\Delta^2\phi
    -\frac12\phi'(\partial_rV)|u|^2\right)dx
    +\Im\int\phi'uB_\tau\cdot\overline{\nabla_Au}dx
    \\
    & \ \ =\Re\int
    f\left(\nabla\phi\cdot\overline{\nabla_Au}+\frac12(\Delta\phi)\overline{u}\right)dx
    \pm\epsilon\Im\int u\nabla\phi\cdot\overline{\nabla_A
    u}dx.
    \nonumber
  \end{align}
  Now identity \eqref{eq:identity} follows by summing up
  \eqref{eq:symmetric} with \eqref{eq:anti-symmetric}. The following
  regularity
  remark completes the proof.
\begin{remark}\label{rem:regularity}
  We must notice that the term requiring more regularity
  on $u$, in order to justify the integration by parts, is
  the one involving second commutator $[H,[H,\phi]]$. In principle, it requires $u\in\mathcal D(H^2)$
  to make sense; actually, the integration by parts on it shows that a term of the
  form $\int\Delta u\nabla\phi\cdot\overline{\nabla_A u}$ needs to be a priori bounded,
  and
  $u\in\mathcal H^{\frac32}$ is sufficient. The proof of identity
  \eqref{eq:identity} for $\mathcal H^1$-solutions follows by
  approximation on $f$. Indeed, if $f\in\mathcal
  D(H^s)$, $s\geq0$, and $\epsilon\neq0$, there
  exists a unique solution $u\in\mathcal D(H^s)$ of
  \eqref{eq:helmholtz}; now the density of $\mathcal C^\infty_0$
  in $\mathcal D(H^s)$ completes the argument.
\end{remark}
\end{proof}

\section{Proof of the main Theorems \ref{thm:main}, \ref{thm:main2}}\label{sec:proof}
We pass now the the proofs of our main theorems. These are based on
identity \eqref{eq:identity}, by suitable choices of the multipliers
$\phi,\varphi$. Our choice of the multipliers follows an idea
introduced in \cite{brv}, and then used in \cite{FV} with explicit
definitions. The multipliers are analogous, in dimensions $n=3$,
$n\geq4$, but give different results and conditions on the
potentials (see Remark \ref{rem:assu}).

\subsection{Proof of Theorem \ref{thm:main}} We denote by $r=|x|$; following
\cite{FV}, we define $\phi_0$ as
  \begin{equation*}
  \phi_0(x)=\int_0^x\phi_0'(s)\,ds,
\end{equation*}
where
\begin{equation*}
  \phi'_0=\phi'_0(r)=
  \begin{cases}
    M+\frac13r,
    \qquad
    r\leq1
    \\
    M+\frac12-\frac1{6r^2},
    \qquad
    r>1,
  \end{cases}
\end{equation*}
and $M$ is given by assumption \eqref{eq:condition}. We have
\begin{equation*}
  \phi_0''(r)=
  \begin{cases}
    \frac13,
    \qquad
    r\leq1
    \\
    \frac1{3r^3},
    \qquad
    r>1
  \end{cases}
\end{equation*}
and the bilaplacian is given by
\begin{equation*}
  \Delta^2\phi_0(r)=-4\pi\delta_{x=0}-\delta_{|x|=1},
\end{equation*}
in the distributional sense. By scaling, for any $R>0$ we define
\begin{equation*}
  \phi_R(r)=R\phi_0\left(\frac rR\right),
\end{equation*}
hence
\begin{equation}\label{eq:fi1}
  \phi'_R(r)=
  \begin{cases}
    M+\frac r{3R},
    \qquad
    r\leq R
    \\
    M+\frac12-\frac{R^2}{6r^2},
    \qquad
    r>R
  \end{cases}
\end{equation}
\begin{equation}\label{eq:fi2}
  \phi''_R(r)=
  \begin{cases}
    \frac1{3R},
    \qquad
    r\leq R
    \\
    \frac1R\cdot\frac{R^3}{3r^3},
    \qquad
    r>R
  \end{cases}
\end{equation}
\begin{equation}\label{eq:filapl}
  \Delta\phi_R(r)=
  \begin{cases}
    \frac1R+\frac{2M}{r},
    \qquad
    r\leq R
    \\
    \frac{1+2M}{r},
    \qquad
    r>R
  \end{cases}
\end{equation}
\begin{equation}\label{eq:fibi}
  \Delta^2\phi_R(r)=-4\pi\delta_{x=0}-\frac1{R^2}\delta_{|x|=R}.
\end{equation}
Observe that $\phi'_R,\phi''_R,\Delta\phi_R\geq0$ and moreover
\begin{equation}\label{eq:estimatephi}
  \sup_{r\geq0}\phi'_R(r)\leq M+\frac12
  \qquad
  \sup_{r\geq0}\phi''_R(r)\leq \frac{1}{3R},
  \qquad
  \sup_{r\geq0}\Delta\phi_R\leq\frac{1+2M}{r},
\end{equation}
\begin{equation}\label{eq:estimatephi2}
  \inf_{r\geq0}\phi'_R(r)\geq M.
\end{equation}
In fact, this choice of $\phi_R$ had been made in the reverse way;
we started from the bi-laplacian, which contains the term
$\delta_{r=R}$ and seems to optimize the size condition
\eqref{eq:condition}, as we see in the following.

Now we define $\varphi_R$ as follows:
\begin{equation}\label{eq:varfi}
  \varphi_R(r)=
  \begin{cases}
    \frac\beta R,
    \qquad
    r\leq R
    \\
    \frac\beta r,
    \qquad
    r>R,
  \end{cases}
\end{equation}
for some $\beta<\frac13$ to be chosen later. The reason of the bound
$1/3$ for $\beta$ will be clear in Section \ref{subsubsec2}. Observe
that
\begin{equation}\label{eq:varfidecad}
  C_0\langle x\rangle^{-1}\varphi_R(r)\leq C\langle x\rangle^{-1},
\end{equation}
for some $C=C(\beta)>0$, and $C_0=C_0(\beta)>0$ such that
$C,C_0\to0$ as $\beta\to0$. Here $\langle x\rangle=(1+|x|^2)^{1/2}$.
By a direct computation we obtain
\begin{equation}\label{eq:deltavarphi}
  \Delta\varphi_R=-\frac{\beta}{R^2}\delta_{|x|=R},
\end{equation}
which is true in the distributional sense.

Let us now put the multipliers $\phi_R,\varphi_R$ in identity
\eqref{eq:identity} and begin to estimate. We start with the
estimate of the right-hand side.

\subsubsection{Estimate of the RHS in
\eqref{eq:identity}}\label{subsubsec1}

By \eqref{eq:estimatephi} and Cauchy-Schwartz, we have
\begin{align}\label{eq:estf}
  & \left|\int f\nabla\phi\cdot\overline{\nabla_A u}\right|
  \leq\left(M+\frac12\right)\sum_{j\in\mathbb
  Z}\int_{C(j)}|f|\cdot|\nabla_A u|
  \\
  &\ \ \ \ \ \ \ \
  \leq\left(M+\frac12\right)\sum_{j\in\mathbb Z}\left(2^{-j-1}\int_{C(j)}|\nabla_A u|^2\right)^{\frac12}
  \left(2^{j+1}\int_{C(j)}|f|^2\right)^{\frac12}
  \nonumber
  \\
  &\ \ \ \ \ \ \ \
  \leq\left(M+\frac12\right)\left(\sup_{R>0}\frac1R\int_{|x|\leq R}|\nabla_A u|^2\right)^{\frac12}
  \sum_{j\in\mathbb Z}\left(2^{j+1}\int_{C(j)}|f|^2\right)^{\frac12}
  \nonumber
  \\
  &\ \ \ \ \ \ \ \
  \leq\alpha|||\nabla_A u|||^2
  +C(\alpha)N(f)^2,
  \nonumber
\end{align}
with $\alpha,C(\alpha)>0$. Analogously, by \eqref{eq:estimatephi}
and \eqref{eq:varfi},
\begin{align}\label{eq:estf2}
  & \left|\int f(\frac12\Delta\phi+\varphi)\overline{u}\right|
  \leq\left(\frac12+M+\beta\right)\sum_{j\in\mathbb
  Z}\int_{C(j)}|f|\cdot\frac{|u|}{|x|}
  \\
  &\ \ \ \ \ \ \ \
  \leq\left(\frac12+M+\beta\right)\sum_{j\in\mathbb Z}\left(2^{-j}\int_{C(j)}\frac{|u|^2}{|x|^2}\right)^{\frac12}
  \left(2^{j}\int_{C(j)}|f|^2\right)^{\frac12}
  \nonumber
  \\
  &\ \ \ \ \ \ \ \
  \leq\left(\frac12+M+\beta\right))\left(\sup_{R>0}\frac1{R^2}\int_{|x|=R}|u|^2d\sigma\right)^{\frac12}
  \sum_{j\in\mathbb Z}\left(2^{j}\int_{C(j)}|f|^2\right)^{\frac12}
  \nonumber
  \\
  &\ \ \ \ \ \ \ \
  \leq\alpha\sup_{R>0}\frac1{R^2}\int_{|x|=R}|u|^2d\sigma
  +C(\alpha)N(f)^2.
  \nonumber
\end{align}
It remains now to estimate the last term at the RHS of
\eqref{eq:identity}. Observe that, multiplying \eqref{eq:helmholtz}
by $u$ in the $L^2$-sense and taking the resulting imaginary parts,
we get (see identity \eqref{eq:symmetric2})
\begin{equation}\label{eq:epsilon}
  \epsilon\int|u|^2dx\leq\int|fu|dx.
\end{equation}
On the other hand, taking the real parts we obtain (see identity
\eqref{eq:symmetric})
\begin{equation*}
  \int|\nabla_Au|^2=-\int V|u|^2+\lambda\int|u|^2-\Re\int f\overline
  u.
\end{equation*}
Hence by assumption \eqref{eq:hardy} we have
\begin{equation}\label{eq:epsilon2}
  \int|\nabla_Au|^2\leq C\left(|\lambda|\int|u|^2+\int|fu|\right).
\end{equation}
As a consequence of \eqref{eq:epsilon} and \eqref{eq:epsilon2}, by
\eqref{eq:estimatephi} we can estimate
\begin{align}\label{eq:estepsilon}
  & \left|\epsilon\int u\nabla\phi\cdot\overline{\nabla_Au}\right|
  \leq
  C|\epsilon|^{1/2}\left(|\lambda|\int|u|^2+\int|fu|\right)^{\frac12}
  \left(\int|fu|\right)^{\frac12}
  \\
  &\ \ \ \ \ \ \ \
  \leq C|\epsilon|^{1/2}\int|fu|+
  C\left(|\epsilon\lambda|\int|fu|\int|u^2|\right)^{\frac12}
  \nonumber
  \\
  &\ \ \ \ \ \ \ \
  \leq C(|\epsilon|+|\lambda|)^{\frac12}\int|fu|
  \nonumber
  \\
  &\ \ \ \ \ \ \ \
  \leq
  C(|\epsilon|+|\lambda|)^{\frac12}\cdot|||u|\lambda|^{1/2}|||\cdot N\left(\frac{f}{|\lambda|^{1/2}}\right)
  \nonumber
  \\
  &\ \ \ \ \ \ \ \
  \leq\alpha|||u|\lambda|^{1/2}|||^2+C(\alpha)(|\epsilon|+|\lambda|)
  \left(N\left(\frac{f}{|\lambda|^{1/2}}\right)\right)^2,
  \nonumber
\end{align}
for $\alpha,C(\alpha)>0$. In conclusion, by \eqref{eq:estf},
\eqref{eq:estf2} and \eqref{eq:estepsilon}, for the right-hand side
of \eqref{eq:identity} we have
\begin{align}\label{eq:estRHS}
  & \left|\Re\int
    f\left(\nabla\phi\cdot\overline{\nabla_Au}+\frac12(\Delta\phi)\overline{u}\right)dx
    +\Re\int f\varphi\overline udx
    \pm\epsilon\Im\int u\nabla\phi\cdot\overline{\nabla_A u}dx\right|
  \\
  &\ \
  \leq\alpha\left(|||\nabla_A
  u|||^2+|||u|\lambda|^{1/2}|||^2+\sup_{R>0}\frac1{R^2}\int_{|x|=R}|u|^2d\sigma\right)
  \nonumber
  \\
  &\ \ \ \ \
  +C(\alpha)\left[N(f)^2+(|\epsilon|+|\lambda|)
  \left(N\left(\frac{f}{|\lambda|^{1/2}}\right)\right)^2\right],
  \nonumber
\end{align}
for arbitrary $\alpha>0$.

Our next step is to prove the positivity of the left-hand side of
\eqref{eq:identity}.

\subsubsection{Positivity of the LHS in \eqref{eq:identity}}\label{subsubsec2}

Let us consider the first term. Since $\phi_R$ is radial, we can
exploit the formula
\begin{equation}\label{eq:tang}
  \nabla_AuD^2\phi_R\overline{\nabla_Au}=
  \phi_R''|\nabla_A^ru|^2+\frac{\phi'_R}{|x|}|\nabla_A^\tau u|^2,
\end{equation}
where $\nabla_A^ru=\nabla_Au\cdot x/|x|$ denotes the radial
component of the distorted gradient and $|\nabla_A^\tau u|$ the
modulus of the tangential component, i.e.
\begin{equation*}
  \nabla_A^\tau u\cdot\nabla_A^ru=0,
  \qquad
  |\nabla_A^\tau u|^2=|\nabla_Au|^2-|\nabla_A^ru|^2.
\end{equation*}
By \eqref{eq:tang}, \eqref{eq:estimatephi} and \eqref{eq:varfi},
since $\beta<1/3$ we estimate
\begin{equation}\label{eq:LHS1}
  \int\nabla_AuD^2\phi_R\overline{\nabla_Au}-\int\varphi_R|\nabla_Au|^2
  \geq
  M\int\frac{|\nabla_A^\tau u|^2}{|x|}
  +\frac{1-3\beta}{3}\cdot\frac1R\int_{|x|\leq
  R}|\nabla_Au|^2.
\end{equation}
For the third term, by \eqref{eq:fibi} and \eqref{eq:deltavarphi} we
have
\begin{equation}\label{eq:LHS2}
  \int\left(-\frac14\Delta^2\phi_R+\frac12\Delta\varphi_R\right)|u|^2dx
  =
  \pi|u(0)|^2+\frac{1-2\beta}{4R^2}\int_{|x|=R}|u|^2d\sigma,
\end{equation}
and again this is a positive term. Now we pass to the terms
containing $\partial_rV$ and $B_\tau$. First observe that, by
splitting $\partial_rV=(\partial_rV)_+-(\partial_rV)_-$ and using
\eqref{eq:estimatephi}, \eqref{eq:estimatephi2}, we obtain
\begin{align}\label{eq:LHS3}
  & -\frac12\int\phi'_R(\partial_rV)|u|^2
  \geq
  \frac
  M2\int(\partial_rV)_-|u|^2-\frac{2M+1}{4}\int(\partial_rV)_+|u|^2
  \\
  & \geq\frac
  M2\int(\partial_rV)_-|u|^2-\frac{2M+1}{4}
  \int_0^\infty d\rho\int_{|x|=\rho}(\partial_rV)_+|u|^2d\sigma
  \nonumber
  \\
  & \geq\frac
  M2\int(\partial_rV)_-|u|^2-\frac{2M+1}{4}
  \sup_{R>0}\left(\frac1{R^2}\int_{|x|=R}|u|^2
  d\sigma\right)\||x|^2(\partial_rV)_+\|_{L^1_rL^\infty(S_r)}.
  \nonumber
\end{align}
Analogously, by \eqref{eq:varfidecad} we have
\begin{align}\label{eq:LHS33}
  & -\int\varphi_RV|u|^2\geq C_0(\beta)\int\langle
  x\rangle^{-1}V_-|u|^2-C(\beta)\int\langle
  x\rangle^{-1}V_+|u|^2
  \\
  &\ \ \
  \geq C_0(\beta)\int\langle
  x\rangle^{-1}V_-|u|^2
  \nonumber
  \\
  &\ \ \ \ \ \ -C(\beta)\sup_{R>0}\left(\frac1{R^2}\int_{|x|=R}|u|^2
  d\sigma\right)\|\langle x\rangle^{-1}|x|^2V_+\|_{L^1_rL^\infty(S_r)}.
  \nonumber
\end{align}

The term containing $B_\tau$ does not have sign in principle; hence,
noticing that
\begin{equation*}
  \left|B_\tau\cdot\nabla_Au\right|=|B_\tau|\cdot|\nabla_A^\tau u|,
\end{equation*}
since $B_\tau$ is a tangential vector, we estimate
\begin{align}\label{eq:LHS4}
  & \Im\int_{\R^n}u\phi'_RB_\tau\cdot\overline{\nabla_Au}\,dx
  \geq-\frac{2M+1}{2}\int_{\R^n}|u|\cdot|B_\tau|\cdot|\nabla_A^\tau
  u|\,dx
  \\
  &\ \ \
  \geq-\frac{2M+1}{2}\left(\int\frac{|\nabla_A^\tau
  u|^2}{|x|}\right)^{\frac12}\left(\int_0^{+\infty}d\rho\int_{|x|=\rho}|x|\cdot|u|^2\cdot|B_\tau|^2
  d\sigma\right)^{\frac12}
  \nonumber
  \\
  &\ \ \
  \geq-\frac{2M+1}{2}\left(\int\frac{|\nabla_A^\tau
  u|^2}{|x|}\right)^{\frac12}\left(\sup_{R>0}\frac1{R^2}\int_{|x|=R}|u|^2
  d\sigma\right)^{\frac12}\||x|^{\frac32}B_\tau\|_{L^2_rL^\infty(S_r)}.
  \nonumber
\end{align}
We are ready now to sum \eqref{eq:LHS1}, \eqref{eq:LHS2},
\eqref{eq:LHS3}, \eqref{eq:LHS33}, and \eqref{eq:LHS4}. Due to the
freedom on the choice of $R$ we can take the supremum over $R$ in
\eqref{eq:LHS1}, \eqref{eq:LHS2}. In order to simplify the reading,
let us introduce the following notations:
\begin{equation*}
  a:=\left(\int\frac{|\nabla_A^\tau
  u|^2}{|x|}\right)^{\frac12};
  \qquad
  b:=\left(\sup_{R>0}\frac1{R^2}\int_{|x|=R}|u|^2
  d\sigma\right)^{\frac12}.
\end{equation*}
Moreover, according to assumption \eqref{eq:condition}, we denote
\begin{equation*}
  C_1:=\||x|^{\frac32}B_\tau\|_{L^2_rL^\infty(S_r)};
\end{equation*}
\begin{equation*}
  C_2:=\||x|^2(\partial_rV)_+\|_{L^1_rL^\infty(S_r)};
\end{equation*}
\begin{equation*}
  C_3:=\|\langle x\rangle^{-1}|x|^2V_+\|_{L^1_rL^\infty(S_r)}.
\end{equation*}
Hence we have obtained
\begin{align}\label{eq:LHS5}
  & \int\nabla_AuD^2\phi_R\overline{\nabla_Au}-\int\varphi_R|\nabla_Au|^2
  +\int\left(-\frac14\Delta^2\phi_R+\frac12\Delta\varphi_R\right)|u|^2
  \\
  &
  -\int\left[\frac12\phi'_R(\partial_rV)+\varphi_RV\right]|u|^2
  +\Im\int_{\R^n}u\phi'_RB_\tau\cdot\overline{\nabla_Au}
  \nonumber
  \\
  &
  \geq \frac{1-3\beta}{3}\sup_{R>0}\left(\frac1R\int_{|x|\leq
  R}|\nabla_Au|^2\right)
  +\pi|u(0)|^2
  \nonumber
  \\
  &\ \ \
  +\frac M2\int(\partial_rV)_-|u|^2
  +C_0(\beta)\int\langle x\rangle^{-1}V_-|u|^2
  \nonumber
  \\
  &\ \ \
  +Ma^2
  -\frac{2M+1}{2}C_1ab
  +\frac14[1-2\beta-(2M+1)C_2-4C(\beta)C_3]b^2.
  \nonumber
\end{align}
Then we need to prove that
\begin{equation*}
  +Ma^2
  -\frac{2M+1}{2}C_1ab
  +\frac14[1-2\beta-(2M+1)C_2-4C(\beta)C_3]b^2>0,
\end{equation*}
for any $a,b$. By homogeneity, it is sufficient to prove that
\begin{equation}\label{eq:algebra}
  +Ma^2
  -\frac{2M+1}{2}C_1a
  +\frac14[1-2\beta-(2M+1)C_2-4C(\beta)C_3]>0,
\end{equation}
for any $a$. Since $\beta$ is arbitrary in the definition
\eqref{eq:varfi} of $\varphi$, we can choose
$\beta\in(-\gamma,\gamma)$, for $\gamma>0$ arbitrarily small. As a
consequence also the constant $C(\beta)$ is arbitrarily small (see
\eqref{eq:varfidecad}) Hence we can neglect the terms containing
$\beta$, $C(\beta)$, and \eqref{eq:algebra} is satisfied if
\begin{equation}\label{eq:conditionproof}
  \frac{\left(M+\frac12\right)^2}{M}C_1^2+2\left(M+\frac12\right)C_2<1,
\end{equation}
which in fact coincides with \eqref{eq:condition}. In conclusion, we
have proved that, under assumption \eqref{eq:condition},
\begin{align}\label{eq:LHS6}
  & \int\nabla_AuD^2\phi_R\overline{\nabla_Au}-\int\varphi_R|\nabla_Au|^2
  +\int\left(-\frac14\Delta^2\phi_R+\frac12\Delta\varphi_R\right)|u|^2
  \\
  &
  -\int\left[\frac12\phi'_R(\partial_rV)+\varphi_RV\right]|u|^2
  +\Im\int_{\R^n}u\phi'_RB_\tau\cdot\overline{\nabla_Au}
  +\lambda\int\varphi_R|u|^2
  \nonumber
  \\
  &
  \geq \frac{1-3\beta}{3}\sup_{R>0}\left(\frac1R\int_{|x|\leq
  R}|\nabla_Au|^2\right)
  +\pi|u(0)|^2
  \nonumber
  \\
  &\ \ \
  +\frac M2\int(\partial_rV)_-|u|^2
  +C_0(\beta)\int\langle x\rangle^{-1}V_-|u|^2
  \nonumber
  \\
  &\ \ \
  +\delta
  \left(
  \int\frac{|\nabla_A^\tau u|^2}{|x|}
  +\sup_{R>0}\frac1{R^2}\int_{|x|=R}|u|^2
  d\sigma
  \right)
  +C_0(\beta)\lambda\int\frac{|u|^2}{\langle x\rangle}
  \geq0,
  \nonumber
\end{align}
if $\lambda\geq0$, for a sufficiently small $\delta>0$ depending on
$B_\tau,(\partial_rV)_+$.

At this point, the proof of Theorem \ref{thm:main} is complete by
\eqref{eq:estRHS} and \eqref{eq:LHS6}, up to choose $\alpha$ in
\eqref{eq:estRHS} sufficiently small; actually one needs to notice
that trivially
\begin{equation*}
  |||\lambda^{\frac12}u|||\leq\lambda\int\frac{|u|^2}{\langle
  x\rangle}.
\end{equation*}

\subsection{Proof of Theorem \ref{thm:main2}}\label{sec:4D}
 The proof in dimension $n\geq4$ is completely analogous to the 3D
 case. We first define the following multipliers:
 \begin{equation*}
  \phi_0(x)=\int_0^x\phi_0'(s)\,ds,
\end{equation*}
where
\begin{equation*}
  \phi'_0=\phi'_0(r)=
  \begin{cases}
    M+\frac{n-1}{2n}r,
    \qquad
    r\leq1
    \\
    M+\frac12-\frac1{2nr^{n-1}},
    \qquad
    r>1,
  \end{cases}
\end{equation*}
and $M>0$ is now an arbitrary constant. Observe that $\phi_0$
coincides exactly with the one introduced in the 3D proof. Again, by
scaling we define
\begin{equation*}
  \phi_R(r)=R\phi_0\left(\frac rR\right),
\end{equation*}
and by direct computations we obtain
\begin{equation}\label{eq:fi14d}
  \phi'_R=\phi'_0\left(\frac rR\right)=
  \begin{cases}
    M+\frac{n-1}{2n}\cdot\frac rR,
    \qquad
    r\leq R
    \\
    M+\frac12-\frac{R^{n-1}}{2nr^{n-1}},
    \qquad
    r>R,
  \end{cases}
\end{equation}
\begin{equation}\label{eq:fi24d}
  \phi''_R=
  \begin{cases}
    \frac{n-1}{2n}\cdot\frac 1R,
    \qquad
    r\leq R
    \\
    \frac{n-1}{2n}\cdot\frac{R^{n-1}}{r^n},
    \qquad
    r>R;
  \end{cases}
\end{equation}
\begin{equation}\label{eq:deltafi4D}
  \Delta\phi_R(r)=
  \begin{cases}
    \frac{n-1}{2R}+\frac{M(n-1)}{r},
    \qquad
    r\leq R
    \\
    \frac{(2M+1)(n-1)}{2r},
    \qquad
    r>R;
  \end{cases}
\end{equation}
moreover, the bilaplacian gives now
\begin{align}\label{eq:bifi4D}
  \Delta^2\phi_R(r)=
  &
  -\frac{n-1}{2R^2}\delta_{|x|=R}
  -M\frac{(n-1)(n-3)}{r^3}\chi_{[0,R]}
  \\
  &
  -\left(M+\frac12\right)\frac{(n-1)(n-3)}{r^3}\chi_{(R,+\infty)},
  \nonumber
\end{align}
in the distributional sense, where $\chi$ denotes the characteristic
function. Observe that also here the bi-laplacian is negative; the
terms involving the characteristic functions turn out to be crucial
in view to improve the 3D condition \eqref{eq:condition} in
\eqref{eq:condition2}. Moreover let us notice that, as in 3D case,
$\phi'_R,\phi''_R,\Delta\phi_R\geq$ and
\begin{equation}\label{eq:estimatefi4d}
  \sup_{r\geq0}\phi'_R(r)\leq M+\frac12,
  \quad
  \sup_{r\geq0}\phi''_R(r)\leq\frac{n-1}{2nR},
  \quad
  \sup_{r>0}\Delta\phi_R(r)\leq\frac{(2M+1)(n-1)}{2r},
\end{equation}
\begin{equation}\label{eq:estimatefi4d2}
  \inf_{r\geq0}\phi'_R\geq M.
\end{equation}
As in \eqref{eq:varfi}, we define
\begin{equation}\label{eq:varfi4d}
  \varphi_R(r)=
  \begin{cases}
    \frac\beta R,
    \quad
    r\leq R
    \\
    \frac\beta r,
    \quad
    r>R
  \end{cases}
\end{equation}
for some $\beta<(n-1)/2n$. Obviously \eqref{eq:varfidecad} is still
true. Moreover we have
\begin{equation}\label{eq:deltavarfi4d}
  \Delta\varphi_R(r)=
  -\frac{\beta}{R^2}\delta_{|x|=R}-\frac{\beta(n-3)}{r^3}\chi_{(R,+\infty)}.
\end{equation}
From now on the proof is almost the same as in the 3D case.

\subsubsection{Estimate of the RHS in \eqref{eq:identity}}
This stuff is identical as in subsection \ref{subsubsec1}. Actually,
with the same argument, by \eqref{eq:estimatefi4d},
\eqref{eq:varfidecad} and assumption \eqref{eq:hardy2} we obtain
\eqref{eq:estRHS}, exactly as in the 3D case. We omit further
details.

\subsubsection{Positivity of the LHS in \eqref{eq:identity}}
Here we have a difference with respect to the 3D case. Indeed, the
two terms involving the characteristic functions in
\eqref{eq:bifi4D} have to be exploited in order to get positivity
with optimal conditions on the potentials.

Let us start again by formula \eqref{eq:tang}; by this,
\eqref{eq:fi14d}, \eqref{eq:fi24d} and \eqref{eq:varfi4d} we easily
see that
\begin{align}\label{eq:LHS14D}
  & \int\nabla_AuD^2\phi_R\overline{\nabla_Au}-\int\varphi_R|\nabla_Au|^2
  \\
  &\ \ \ \ \ \
  \geq
  M\int\frac{|\nabla_A^\tau u|^2}{|x|}
  +\left(\frac{n-1}{2n}-\beta\right)\cdot\frac1R\int_{|x|\leq
  R}|\nabla_Au|^2,
  \nonumber
\end{align}
for any $R>0$. This terms is positive since $\beta<(n-1)/2n$. By
\eqref{eq:bifi4D} and \eqref{eq:deltavarfi4d} we get
\begin{align}\label{eq:termini}
  & \frac14\Delta^2\phi_R-\frac12\Delta\varphi_R
  = \Delta\left(\frac14\Delta\phi_R-\frac12\varphi_R\right)
  \\
  &
  =-\frac{n-1-4\beta}{8}\cdot\frac1{R^2}\delta_{|x|=R}
  -\frac{M(n-1)(n-3)}{4r^3}\chi_{[0,R]}
  \nonumber
  \\
  &
  \ \ -\frac{(2M+1)(n-1)(n-3)-4\beta(n-3)}{8r^3}\chi_{(R,+\infty)};
  \nonumber
\end{align} As a consequence
\begin{align}\label{eq:LHS24D}
  &
  \int\left(-\frac14\Delta^2\phi_R+\frac12\Delta\varphi_R\right)|u|^2
  \\
  &
  \geq
  \frac{n-1-4\beta}{8R^{2}}\int_{|x|=R}|u|^2d\sigma
  +\left(\frac{M(n-1)(n-3)}{4}-K(\beta)\right)\int\frac{|u|^2}{|x|^3}dx,
  \nonumber
\end{align}
with $0\leq K(\beta)\to0$ as $\beta\to0$; this term is positive, up
tho choose $\beta$ small enough. As in the previous case, we now
observe that, by \eqref{eq:estimatefi4d}
\begin{align}\label{eq:LHS34D}
  & -\frac12\int\phi'_R(\partial_rV)|u|^2
  \geq
  \frac
  M2\int(\partial_rV)_-|u|^2-\frac{2M+1}{4}\int(\partial_rV)_+|u|^2
  \\
  & \geq\frac
  M2\int(\partial_rV)_-|u|^2
  -\frac{2M+1}{4}
  \||x|^3(\partial_rV)_+\|_{L^\infty}\int\frac{|u|^2}{|x|^3}dx,
  \nonumber
\end{align}
\begin{align}\label{eq:LHS334D}
  & -\int\varphi_RV|u|^2\geq C_0(\beta)\int\langle
  x\rangle^{-1}V_-|u|^2-C(\beta)\int\langle
  x\rangle^{-1}V_+|u|^2
  \\
  &\ \ \
  \geq C_0(\beta)\int\langle
  x\rangle^{-1}V_-|u|^2
  -C(\beta)\|\langle x\rangle^{-1}|x|^3V_+\|_{L^\infty}\int\frac{|u|^2}{|x|^3}dx.
  \nonumber
\end{align}
With a similar computation, for the term involving $B_\tau$ we
estimate
\begin{align}\label{eq:LHS44D}
  & \Im\int_{\R^n}u\phi'_RB_\tau\cdot\overline{\nabla_Au}\,dx
  \geq-\frac{2M+1}{2}\int_{\R^n}|u|\cdot|B_\tau|\cdot|\nabla_A^\tau
  u|\,dx
  \\
  &\ \ \
  \geq-\frac{2M+1}{2}\left(\int\frac{|\nabla_A^\tau
  u|^2}{|x|}\right)^{\frac12}\left(\int\frac{|u|^2}{|x|^3}\right)^{\frac12}
  \||x|^2B_\tau\|_{L^\infty}.
  \nonumber
\end{align}
Now we can sum \eqref{eq:LHS14D}, \eqref{eq:LHS24D},
\eqref{eq:LHS34D}, \eqref{eq:LHS334D} and \eqref{eq:LHS44D}, taking
the supremum over $R$; we denote by
\begin{equation*}
  a:=\left(\int\frac{|\nabla_Au|^2}{|x|}\right)^{\frac12};
  \qquad
  b:=\left(\int\frac{|u|^2}{|x|^3}\right)^{\frac12},
\end{equation*}
and according to assumption \eqref{eq:condition2}
\begin{equation*}
  \||x|^2B_\tau\|_{L^\infty}\leq C_1,
\end{equation*}
\begin{equation*}
  \||x|^3(\partial_rV)_+\|_{L^\infty}\leq C_2,
\end{equation*}
\begin{equation*}
  \|\langle x\rangle^{-1}|x|^3V_+\|_{L^\infty}:=C_3<\infty.
\end{equation*}
We obtain
\begin{align}\label{eq:LHS54D}
  & \int\nabla_AuD^2\phi_R\overline{\nabla_Au}-\int\varphi_R|\nabla_Au|^2
  +\int\left(-\frac14\Delta^2\phi_R+\frac12\Delta\varphi_R\right)|u|^2
  \\
  &
  -\int\left[\frac12\phi'_R(\partial_rV)+\varphi_RV\right]|u|^2
  +\Im\int_{\R^n}u\phi'_RB_\tau\cdot\overline{\nabla_Au}
  \nonumber
  \\
  &
  \geq \left(\frac{n-1}{2n}-\beta\right)|||\nabla_Au|||^2
  +\frac{n-1-4\beta}{8}\sup_{R>0}\left(\frac{1}{R^{2}}\int_{|x|=R}|u|^2d\sigma\right)
  \nonumber
  \\
  &
  \ \ \ +\frac M2\int(\partial_rV)_-|u|^2
  +C_0(\beta)\int\langle x\rangle^{-1}V_-|u|^2
  \nonumber
  \\
  &
  \ \ \ +Ma^2-\left(M+\frac12\right)C_1ab
  \nonumber
  \\
  &
  \ \ \ +\frac14\left[M(n-1)(n-3)-(2M+1)C_2-4C(\beta)C_3-4K(\beta)\right]b^2.
  \nonumber
\end{align}
It remains to prove that
\begin{align*}
  & Ma^2-\left(M+\frac12\right)C_1ab
  \\
  &
  \ \ \ +\frac14\left[M(n-1)(n-3)-(2M+1)C_2-4C(\beta)C_3-4K(\beta)\right]b^2>0,
\end{align*}
for any $a,b$. Again, by homogeneity it is sufficient to show that
\begin{align*}
  & Ma^2-\left(M+\frac12\right)C_1a
  \\
  &
  \ \ \ +\frac14\left[M(n-1)(n-3)-(2M+1)C_2-4C(\beta)C_3-4K(\beta)\right]>0,
\end{align*}
for any $a$. This is satisfied if
\begin{equation*}
  \frac{1}{(n-1)(n-3)}\left[\frac{\left(M+\frac12\right)^2}{M^2}C_1^2
  +2\frac{\left(M+\frac12\right)}{M}C_2\right]<1.
\end{equation*}
Finally, notice that
\begin{equation*}
  \inf_{M>0}\frac{\left(M+\frac12\right)^2}{M^2}
  =\inf_{M>0}\frac{\left(M+\frac12\right)}{M}=1
\end{equation*}
and the infimum is reached in the limit as $M\to\infty$. Since $M$
is arbitrary in the definition of $\phi_R$ we can optimize in terms
of $C_1,C_2$, and conclude that the last condition is
\begin{equation}\label{eq:conditionfinal}
  C_1^2+2C_2<(n-1)(n-3),
\end{equation}
which is in fact assumption \eqref{eq:condition2}. In conclusion,
assumption \eqref{eq:condition2} implies that
\begin{align}\label{eq:LHS64D}
  & \int\nabla_AuD^2\phi_R\overline{\nabla_Au}-\int\varphi_R|\nabla_Au|^2
  +\int\left(-\frac14\Delta^2\phi_R+\frac12\Delta\varphi_R\right)|u|^2
  \\
  &
  -\int\left[\frac12\phi'_R(\partial_rV)+\varphi_RV\right]|u|^2
  +\Im\int_{\R^n}u\phi'_RB_\tau\cdot\overline{\nabla_Au}
  +\lambda\int\varphi_R|u|^2
  \nonumber
  \\
  &
  \geq \left(\frac{n-1}{2n}-\beta\right)|||\nabla_Au|||^2
  +\frac{n-1}{8}\sup_{R>0}\left(\frac{1}{R^{2}}\int_{|x|=R}|u|^2d\sigma\right)
  \nonumber
  \\
  &
  \ \ \ \int(\partial_rV)_-|u|^2
  +C_0(\beta)\int\langle x\rangle^{-1}V_-|u|^2
  \nonumber
  \\
  &
  \ \ \ +\delta
  \left(
  \int\frac{|\nabla_A^\tau u|^2}{|x|}
  +\int\frac{|u|^2}{|x|^3}dx
  \right)
  +C_0(\beta)\lambda\int\frac{|u|^2}{\langle x\rangle}
  \geq0,
  \nonumber
\end{align}
if $\lambda\geq0$, for a sufficiently small $\delta>0$ depending on
$B_\tau,(\partial_rV)_+$. The proof of Theorem \ref{thm:main2} is
complete by \eqref{eq:estRHS} and \eqref{eq:LHS64D}, up tho choose
$\alpha>0$ sufficiently small in \eqref{eq:estRHS}.

\end{document}